\documentclass[a4paper,10pt]{article}
\usepackage[english]{babel}
\usepackage[utf8]{inputenc}
\usepackage[T1]{fontenc}
\usepackage{amsmath,amssymb,amsthm}

\newtheorem{theorem}{Theorem}
\newtheorem{coro}{Corollary}
\newtheorem{lemma}{Lemma}

\newtheorem*{remarks}{Remark}

\title{Characterizations of canonically compactifiable graphs via intrinsic metrics and algebraic properties}
\author{Simon Puchert}

\begin{document}

\maketitle

\begin{abstract}
We consider infinite graphs and the associated energy forms. We show
that a graph is canonically compactifiable (i.e. all functions of
finite energy are bounded) if and only if the underlying set is
totally bounded with respect to any finite measure intrinsic metric.
Furthermore, we show that a graph is canonically compactifiable if and only if the space of functions of finite energy is an algebra.
These results answer questions in a recent work of Georgakopoulos,
Haeseler, Keller, Lenz, Wojciechowski.
\end{abstract}

\section*{Introduction}
Open bounded sets in Euclidean space provide an important and well
studied instance of spectral geometry. Recently, discrete analogues
of such sets have become a focus of attention
\cite{GHKLW,KLWS,KLSS}. In particular, \cite{GHKLW} proposes the concept
of canonically compactifiable graphs as graphs with strong
intrinsic compactness properties. By definition, a graph is called
canonically compactifiable if all functions of finite energy are
bounded. For such graphs, there is a natural compactification, namely, the Royden compactification.
Clearly, in order to study the geometry of such graphs, it is
desirable to understand metric compactness features of such graphs.

As shown in \cite{GHKLW}, total boundedness with respect to common
metrics such as the resistance metric or a standard length metric
implies that the underlying graph is canonically compactifiable but
the converse is not true. So, \cite{GHKLW} leaves open the question
of a metric compactness characterization of such graphs. Still, a
candidate for such a characterization is proposed there. More
specifically, it is shown that  total boundedness with respect to
all finite measure intrinsic metrics implies that the graph is
canonically compactifiable and the converse is shown to hold for
locally finite trees. The converse for general graphs, however,
remained open and is posed as a problem in \cite{GHKLW}.  The first main
result of this note (Theorem \ref{main}) solves this problem.
Combining this main result with the mentioned result
of \cite{GHKLW}, we obtain that a graph is canonically compactifiable
if and only if it is totally bounded with respect to all metrics that are
intrinsic with respect to a finite measure (Corollary \ref{coro-main}).
This characterization turns out to extend without modification to general (non-connected) graphs.

To put this result in perspective, we briefly discuss the relevance of
intrinsic metrics next. Intrinsic metrics for strongly local
Dirichlet forms were introduced in \cite{Sturm} and have
subsequently played a fundamental role as they allow for a study
of intrinsic spectral geometry of such forms. For general regular
Dirichlet forms a concept was only proposed recently in \cite{FLW}
(see \cite{Fol1,Fol2,GHM,Uem} for independent related
studies on graphs). A special feature of the case of general
Dirichlet forms is well worth pointing out: While in the strongly
local case there is a maximal intrinsic metric there are, in general,
several incomparable intrinsic metrics on graphs \cite{FLW}. Hence,
in general, one cannot expect that it will be sufficient to consider
only one intrinsic metric for graphs. Recent years have witnessed
rather successful applications of intrinsic metrics in order to
understand spectral geometry of graphs, see the mentioned works as
well as e.g. \cite{BHK,BKW,HKW}. Given this, it is very natural to
look for a characterization of the intrinsic compactness property
of canonical compactifiability in terms of intrinsic metrics. Theorem \ref{main} provides such a characterization.

In the last section we provide an answer to another question raised in \cite{GHKLW}. This question concerns an algebraic characterization of canonically compactifiable graphs. More specifically, in \cite{GHKLW} it is shown that the set of functions of finite energy is an algebra if the graph is canonically compactifiable and we show that the converse is also true (Theorem~\ref{algebraic}). Our proof can be modified slightly to obtain a similar characterization for uniform transience (Theorem~\ref{theorem:ut}), a  property that was recently introduced in \cite{KLWS}. Moreover, it can also be generalized to resistance forms in the sense of Kigami \cite{Kig}. We briefly discuss this extension in Theorem~\ref{rf}. Typical examples for resistance forms are provided by metric graphs and fractals, we refer to \cite{Kig,Kig2} for details.  

\bigskip

\textbf{Acknowledgments.} The author is grateful to Daniel Lenz, Marcel Schmidt and Rados{\l}aw K. Wojciechowski for enlightening discussions, pointing out the
problems and help in preparing the manuscript. Further thanks go to two anonymous referees for helpful suggestions.

\section{Background}
In this section we first introduce the necessary notations and recall
basic facts shown in \cite{GHKLW} (see \cite{KL1} for a description of
the general setting as well).

A weighted graph $G=(X,b)$ consists of a nonempty countable set $X$ of nodes and a symmetric edge weight function $b: X\times X \rightarrow [0,\infty)$ that vanishes on the diagonal and satisfies the summability condition $$ \sum\limits_{y\in X} b(x,y) < \infty\ \mathrm{for\ all\ } x\in X.$$ Two nodes $x,y\in X$ are called \emph{connected} if there is a sequence $(x=x_0,\ldots,x_n=y)$ with $b(x_k,x_{k+1}) > 0$ for all $0\leq k < n$. Similarly, a graph is called connected, if all of its nodes are connected.

This graph structure gives rise to a quadratic form that assigns to any function $f:X\rightarrow \mathbb{C}$ its \emph{Dirichlet energy}
$$ \tilde{Q}(f) := \frac{1}{2}\sum\limits_{x,y\in X} b(x,y)|f(x)-f(y)|^2 $$
and consequently defines the \emph{functions of finite energy}
$$\mathfrak{D}(G) := \{f:X\rightarrow \mathbb{C} \mid \tilde{Q}(f) < \infty\}. $$
This set is closed with respect to addition, since
$$ \tilde{Q}(f+g)^{1/2} \leq \tilde{Q}(f)^{1/2} + \tilde{Q}(g)^{1/2}. $$

The graph $(X,b)$ is called \emph{canonically compactifiable}, if all functions of finite energy are bounded, that is if $\mathfrak{D}(G) \subseteq \ell^\infty(X)$.

In the rest of this note we will only look at connected graphs. We can do this since a graph is canonically compactifiable if and only if it has finitely many connected components (i.e. equivalence classes with respect to connectedness) and every connected component is canonically compactifiable. We will explicitly state if we don't use this assumption.

For any $o\in X$, we define a pseudo-norm $\|\cdot\|_o$ on $\mathfrak{D}(G)$ via 
$$ \|f\|_o^2 := \tilde{Q}(f) + |f(o)|^2. $$
Since we assumed connectedness, this yields a Hilbert space $(\mathfrak{D}(G), \|f\|_o)$ and the pointwise evaluation
$$\mathfrak{D}(G) \rightarrow \mathbb{C}, \qquad f\mapsto f(x)$$
is continuous for every $x\in X$, see e.g. \cite{S}, Section 1.2.

Pseudo metrics are symmetric functions $\sigma: X\times X\rightarrow [0,\infty)$ that vanish on the diagonal and satisfy the triangle inequality $\sigma(x,z)\leq \sigma(x,y) + \sigma(y,z)$.

Any pseudo metric $\sigma$ naturally induces a distance from any nonempty subset $U\subseteq X$ via $$\sigma(\cdot,U): X\rightarrow[0,\infty),\qquad \sigma(x, U) = \inf\limits_{y\in U} \sigma(x,y)$$ and the \emph{diameter} of $U\subseteq X$ by $$\mathrm{diam}_\sigma(U) := \sup\limits_{x,y \in U} \sigma(x,y).$$

Whenever $(X,b)$ is a graph and $m$ is a measure on $X$ (i.e. an additive function $\mathcal{P}(X)\rightarrow [0,\infty]$ induced by a node weight function $X\rightarrow (0,\infty)$), a pseudo metric $\sigma$ is called \emph{intrinsic with respect to the measure} $m$, if the inequality $$\frac{1}{2}\sum\limits_{y\in X} b(x,y)\sigma(x,y)^2 \leq m(\{x\})$$ holds for all $x\in X$.

Any graph $(X,b)$ comes with a metric $\varrho$ defined as
$$\varrho(x,y):=\sup\left\{|f(x)-f(y)| \mid f\in \mathfrak{D}(G)\ \mathrm{with}\ \tilde{Q}(f) \leq 1\right\}.$$
Given this definition, it is easy to see that  any function $f$ of
bounded Dirichlet energy satisfies
$$|f(x)-f(y)| \leq \tilde{Q}(f)^{1/2} \varrho(x,y).$$ Indeed, the inequality is optimal; the definition of $\varrho$ gives that it is characterized by $$ \inf\{\tilde{Q}(f)\mid |f(x) - f(y)| = C\} = \frac{C^2}{\varrho(x,y)^2}$$ for any $C>0$.
The metric $\varrho$ is tied to canonical compactifiability,
as Theorem~4.3 in \cite{GHKLW} proves that a connected graph $(X,b)$ is
canonically compactifiable if and only if it is bounded with respect
to~$\varrho$, i.e. $\mathrm{diam}_\varrho (X) < \infty$. This will be used below.

\begin{remarks}
 This metric is closely tied to the resistance metric $r$
 by $\varrho^2 = r$ (it is shown that $r$ is a metric for locally finite graphs in \cite{GHKLW}, Theorem 3.19 and for general graphs in \cite{LSS}).
The metric is also related to intrinsic metrics (see
Theorem 3.13 in \cite{GHKLW}). 
\end{remarks}

\begin{remarks}
 Let us emphasize that our results do not assume
local finiteness of the graph.
\end{remarks}

\section{Characterization via intrinsic metrics}
In this section, we provide a characterization of canonical compactifiability via intrinsic pseudo metrics.

A key step is the subsequent lemma. It provides an estimate for the energy of the distance to a set with respect to an intrinsic pseudo metric, which may be of interest in other contexts as well. A weaker bound (by $m(X)$ instead of $2m(X\setminus U)$) is given in \cite{GHKLW}, Propositions 3.10 and 3.11.

\begin{lemma} \label{estimate}
 Let $G = (X,b)$ be a graph and let $\sigma$ be an intrinsic pseudo metric with respect to a finite measure $m$ on $X$. For a nonempty subset $U\subset X$, the energy of $\sigma(\cdot,U)$ is bounded by
 $$\tilde{Q}(\sigma(\cdot,U)) \leq 2 m(X\setminus U).$$
\end{lemma}

\begin{proof}
 Immediately, we deduce $\sigma(x,U) = 0$ for all $x\in
U$ and $|\sigma(x,U)- \sigma(y,U)|\leq \sigma(x,y)$ for all $x,y\in X$.
These properties already imply the desired bound on the Dirichlet energy of $\sigma(\cdot,U)$:

\begin{eqnarray*}
\tilde{Q}(\sigma(\cdot,U))  &= & \frac{1}{2} \sum_{x,y} b(x,y) (\sigma(x,U) - \sigma(y,U))^2\\
& = & \frac{1}{2} \sum_{(x,y) \in X^2\setminus U^2} b(x,y) (\sigma(x,U) - \sigma(y,U))^2\\
& \leq & \frac{1}{2} \sum_{(x,y) \in X^2\setminus U^2} b(x,y) \sigma(x,y)^2\\
& = & 
\sum_{x\in X\setminus U} \left(\frac{1}{2} \sum_{y\in X} b(x,y) \sigma(x,y)^2\right) +  
\sum_{y\in X\setminus U} \left(\frac{1}{2} \sum_{x\in U} b(x,y) \sigma(x,y)^2\right)\\
&\leq & 2 m (X\setminus U).
\end{eqnarray*}
Here, we used the fact that $\sigma$ is intrinsic with respect to $m$ in the
last estimate.
\end{proof}

\begin{theorem}\label{main}
Let $(X,b)$ be a canonically compactifiable graph and consider a pseudo
metric $\sigma$ which is intrinsic with respect to a finite measure $m$.
Then $(X,b)$ is totally bounded with respect to $\sigma$.
\end{theorem}

\begin{proof}
 Fix an arbitrary $\varepsilon > 0$. We have to find a finite subset
 $S$ of $X$ with $\sigma (x,S) < \varepsilon$ for all $x\in X$.

For  $\delta > 0$, define the set $U_\delta = \{x\in X: m(\{x\}) <
\delta\}$. As $m$ is finite and the graph is canonically
compactifiable (i.e. $\mathrm{diam}_\varrho (X) < \infty$ holds by
the discussion above), we can choose $\delta > 0$ such that
$$m(U_\delta) < \frac{\varepsilon^2}{2\ \mathrm{diam}_\varrho(X)^2}.$$

We now claim that the set $S:=X\setminus U_\delta$ has the desired
properties:

Indeed, the set is finite as we clearly have $|S|
\leq \frac{m(X)}{\delta} < \infty$.

Moreover, $\sigma(x, S) < \varepsilon$ holds as
can be seen as follows: Consider the function
$\sigma(\cdot,S)$. Lemma \ref{estimate} helps us estimate the Dirichlet energy of this function:
\begin{eqnarray*}
\tilde{Q}(\sigma(\cdot,S)) &\leq & 2 m (U_\delta)\\
& < & \frac{\varepsilon^2}{\mathrm{diam}_\varrho(X)^2}.
\end{eqnarray*}

Now, recall the inequality $|f(x)-f(y)|\leq \tilde{Q}(f)^{1/2}\varrho(x,y)$ and
pick an arbitrary point $o\in S$ to see
$$\sigma(x, S) \leq |f(x) - f(o)| \leq \tilde{Q}(\sigma(\cdot, S))^{1/2} \mathrm{diam}_\varrho(X) < \varepsilon. $$ This
finishes the proof.
\end{proof}

Combining the previous theorem with its converse, \cite{GHKLW}, Corollary 4.5, and the fact that a general graph is canonically compactifiable if and only if it has finitely many connected components and each component is canonically compactifiable, we
obtain the following characterization of canonically compactifiable
graphs.

\begin{coro}\label{coro-main}
 A (not necessarily connected) graph $(X,b)$ is canonically compactifiable if and only if $X$
is totally bounded with respect to any pseudo metric $\sigma$ that is intrinsic with respect to a finite measure.
\end{coro}

\section{Algebraic characterization}

In this section, we will prove that a graph is canonically compactifiable if and only if the space of functions of finite energy is an algebra (with the usual pointwise addition and multiplication of functions). The proof can be transferred to show a similar algebraic characterization of uniform transience and can even be extended to resistance forms, see the discussion after Corollary~\ref{coro:algebra}. 

Since Lemma 4.8 in \cite{GHKLW} already states that the space of functions of finite Dirichlet energy is an algebra if the underlying graph is canonically compactifiable, we will focus on the other direction. 

Again, we use the splitting of canonical compactifiability to reduce the problem to connected graphs.

\begin{theorem} \label{algebraic}
 Let $G=(X,b)$ be a graph. If the space of functions of finite Dirichlet energy $\mathfrak{D}(G)$ is an algebra, the graph $G$ is canonically compactifiable.
\end{theorem}
\begin{proof}
 We analyze graphs that are not canonically compactifiable and find functions $f\in\mathfrak{D}(G)$ such that $\tilde{Q}(f^2) = \infty$, implying that $\mathfrak{D}(G)$ cannot be an algebra.
 
 Let $G$ be a graph that is not canonically compactifiable and fix an arbitrary node $o\in X$. We know that $\varrho$ is unbounded on $G$, since $G$ is not canonically compactifiable, see Theorem~4.3 in \cite{GHKLW}. Select an infinite sequence of nodes $(x_n:n\in I \subseteq\mathbb{N})$ such that $8^n<\varrho(x_n,o)\leq 8^{n+1}$ (if there is no such $x_n$ for a certain $n$, just omit this index). Now, we aim to find functions $f_n\in \mathfrak{D}(G)$ that satisfy
 $$f_n(o)=0, f_n(x_n) = 4^n, 0\leq f_n\leq 4^n\ \mathrm{and}\ \tilde{Q}(f_n)\leq 4^{-n}.$$
 Such functions exist, because
 $$8^n < \varrho(x_n,o) = \sup \{|f(x_n) - f(o)| \mid f \in \mathfrak{D}(G), \tilde{Q}(f)\leq 1\}$$
 yields
 $$ \inf\{\tilde{Q}(f)\mid f(o)=0, f(x_n) = 4^n, 0\leq f\leq 4^n\} = \frac{(4^n)^2}{\varrho(x_n,o)^2} < 4^{-n}, $$
 where the additional condition $0\leq f\leq 4^n$ can be introduced since $\tilde{Q}$ is Markovian, i.e. $Q(0\vee f \wedge 4^n)\leq Q(f)$ for all $f$.
 
 Considering $\sum\limits_{n\in I}\tilde{Q}(f_n)^{1/2} \leq \sum\limits_{n=1}^{\infty} 2^{-n} = 1$ and the fact that $(\mathfrak{D}(G), \|\cdot\|_o)$ is a Hilbert space, the function $f:= \sum\limits_{n\in I} f_n \in \mathfrak{D}(G)$ is well-defined and $Q(f)\leq 1$. Conversely, for all $n\in I$ we have $$ \tilde{Q}(f^2) \geq \frac{|f(x_n)^2 - f(o)^2|^2}{\varrho(x_n,o)^2} \geq \frac{f_n(x_n)^4}{\varrho(x_n,o)^2}\geq \frac{256^n}{64^{n+1}} = 4^{n-3}, $$
 thus $\tilde{Q}(f^2)=\infty$.
\end{proof}

Combining Theorem \ref{algebraic} with its converse, Lemma 4.8 in \cite{GHKLW}, then yields the following characterization.

\begin{coro}\label{coro:algebra}
 A graph $G=(X,b)$ is canonically compactifiable if and only if the space of functions of finite Dirichlet energy $\mathfrak{D}(G)$ is an algebra.
\end{coro}

Let $\mathfrak{D}_0(G)$ be the closure of $C_c(X)$ in the Hilbert space $$(\mathfrak{D}(G), \|\cdot\|_o)$$ and define the metric
$$\varrho_0(x,y) = \sup\left\{|f(x)-f(y)| \mid f\in \mathfrak{D}_0(G)\ \mathrm{with}\ \tilde{Q}(f) \leq 1\right\}.$$
The graph $G$ is called {\em uniformly transient} if $\mathfrak{D}_0(G) \subseteq C_0(X)$, where $C_0(X)$ stands for the closure of $C_c(X)$ in $\ell^\infty(X)$, see Section~2 in \cite{KLWS}. Moreover, for connected graphs, uniform transience is equivalent to the boundedness of $\varrho_0$, see \cite{KLWS}, Theorem~3.2. 

The space $\mathfrak{D}_0(G) \cap \ell^\infty(X)$ is an algebra, see \cite{Soa}, Theorem~6.2, and the  proof of Theorem~\ref{algebraic} can be modified by replacing $\mathfrak{D}(G)$ with $\mathfrak{D}_0(G)$ and $\varrho$ with $\varrho_0$. These observations yield the following characterization of uniform transience. 
\begin{theorem}\label{theorem:ut}
 A graph is uniformly transient if and only if $\mathfrak{D}_0(G)$ is an algebra.
\end{theorem}

The same line of reasoning also applies to resistance forms in the sense of Kigami. Here we only state the result and sketch a proof. For notation, further background and examples we refer to \cite{Kig,Kig2}.

Let $(\mathcal E, \mathcal F)$ be a resistance form on the set $X \neq \emptyset$, see Definition~2.3.1 in \cite{Kig}, and let 
$$\varrho_{\mathcal E}(x,y) = \sup\{|f(x) - f(y)| \mid f \in \mathcal F, \mathcal E(f) \leq 1\} $$
be the square root of the associated resistance metric. Our main result for resistance forms reads as follows. 

\begin{theorem}\label{rf}
The following assertions are equivalent.
\begin{itemize}
 \item[(i)] $\mathcal F \subseteq \ell^\infty(X)$.
 \item[(ii)] $\rho_{\mathcal E}$ is bounded. 
 \item[(iii)] $\mathcal F$ is an algebra.
\end{itemize}
\end{theorem}
\begin{proof}
 The equivalence of (i) and (ii) can be proven along the same lines as Theorem~4.3 in \cite{GHKLW}. 
 
 (i) $\Longrightarrow$ (iii): This follows from the fact that $\mathcal F \cap \ell^\infty(X)$ is an algebra, see e.g. Theorem~2.22 in \cite{S1} and the following remark.
 
 (iii)  $\Longrightarrow$ (ii): This can be proven  along the same lines as Theorem~\ref{algebraic}.
\end{proof}

\end{document}